\documentclass[a4paper,11pt, preprint]{elsarticle}

\usepackage{amsfonts}
\usepackage{mathrsfs}
\usepackage{amssymb}
\usepackage{amsmath}
\usepackage{amsthm}
\usepackage{lineno,hyperref,enumerate}

\newcommand\supp{\hbox{supp}}
\newcommand\wt{\widetilde}
\newcommand\wh{\widehat}

\newcommand\tp{\top}

\newcommand\Z{\mathbb{Z}}
\newcommand\R{\mathbb{R}}
\newcommand\N{\mathbb{N}}
\newcommand\C{\mathbb{C}}
\newcommand\T{\mathbb{T}}

\newcommand\Sd{\mathcal{S}} 
\newcommand\Dm{\mathcal{D}} 
\newcommand\Id{\mathcal{I}} 

\newcommand\ga{\alpha}
\newcommand\gb{\beta}
\newcommand\gga{\gamma}

\newtheorem{theorem}{Theorem}

\newtheorem{corollary}[theorem]{Corollary}
\newtheorem{example}{Example}

\begin{document}
\title{Linear multiscale transforms based on even-reversible  subdivision operators }

\author[rvt0]{Nira Dyn}
\ead{niradyn@post.tau.ac.il}
\address[rvt0]{School of Mathematical Sciences, Tel-Aviv University, Israel}

\author[rvt2]{Xiaosheng Zhuang\corref{cor}\fnref{fn1}}
\ead{xzhuang7@cityu.edu.hk}
\address[rvt2]{Department of Mathematics, City University of Hong Kong, Tat Chee Avenue, Kowloon Tong, Hong Kong}
\cortext[cor]{Corresponding author.}


\fntext[fn1]{The research  of X. Z. and the work described in this paper was partially supported by a grant from the Research Grants Council of the Hong Kong Special Administrative Region, China (Project No.  CityU 11300717) and grants from City University of Hong Kong (Project No.: 7200462 and 7004445).}

\begin{abstract}
Multiscale transforms for real-valued data, based on interpolatory
subdivision operators have been studied in recent year.  They are easy to define, and can be  extended
to other types of data, for example to  manifold-valued data.
In this paper we define linear multiscale transforms, based on certain linear,
non-interpolatory  subdivision operators, termed ``even-reversible''.  For such operators, we prove, using Wiener's lemma, the existence of an inverse to the linear operator defined by the even part of the subdivision mask, and termed it ``even-inverse''. We show that the non-interpolatory subdivision operators,  with
spline  or with   pseudo-spline masks, are even-reversible,
and derive explicitly, for the quadratic and cubic spline subdivision operators, the symbols of the corresponding even-inverse operators.  We also analyze properties of the multiscale transforms based on even-reversible subdivision operators, in particular, their stability and the rate of decay of the details.
\end{abstract}


\maketitle
\makeatletter \@addtoreset{equation}{section} \makeatother


\section{Introduction}
\label{sec:intro}

A multiscale transform  is an ubiquitous way for representing data
of the form
$\{(t_k, f_k), \ k\in\Z\}$, with $t_k=k h$ for some positive $h$,
and $f_k$ a real value associated with the point $t_k$. A general approach to
multiscale transforms was introduced in \cite{Harten}.
Such a linear representation can be obtained by discrete wavelet transforms
(see e.g. \cite{Daub:book}). Another common way for generating a multiscale transform
is based on  linear or non-linear interpolatory
subdivision schemes (see e.g. \cite{DonYu,Amat, DG2011}). For a detailed analysis of
several non-linear multiscale transforms see \cite{DynOswald} and
references therein.
The main applications of multiscale transforms are in data compression and
denoising (see e.g. \cite{DonYu,Amat}).

Multiscale transforms based on interpolatory subdivision operators are
simple and rather intuitive. The linear ones were extended to manifold-valued
data in \cite{GW2009,GW2012} and \cite{Rahman}.

In this work we extend the construction
of multiscale transforms from interpolatory subdivision operators to a wider
class of subdivision operators termed {\em even-reversible} operators,
and show that a large class of the subdivision operators studied in the literature are even-reversible.
We derive properties of the multiscale transforms based on even-reversible
subdivision operators, such as decay rates of the data generated by the
transforms, and the stability of the transforms.

\subsection{Multiscale transform: from interpolatory to even-reversible subdivision}
\label{sec:linear-intp}

First we present the multiscale transform based on an interpolatory
subdivision operator $\Sd$.

Let ${\bf f}$ denote a bi-infinite sequence with
elements  $\{ f_k\in\R,\ k\in\Z\}$. Since $\Sd$ is interpolatory,
$$(\Sd{\bf f})_{2k}=f_k,\quad k\in\Z, $$
and it is straightforward to
decimate data at level $j$, ${\bf f}^{j}$, to that at the
coarser level $j-1$,
by taking every
second element. Then the refinement of ${\bf f}^{j-1}$ by $\Sd$ is exact for
the even elements of ${\bf f}^j$.

Here is the pyramid of data generated by $j$ {\em decomposition} steps of
the multiscale transform based on $\Sd$.
\begin{equation}
\label{decomposition}
{\bf f}^{\ell-1}=\{f^{(\ell-1)}_k=f^{(\ell)}_{2k},\  k\in\Z\},\ \  {\bf d}^\ell
={\bf f}^\ell-\Sd{\bf f}^{\ell-1},\ \  \ell=j,j-1,\ldots,1.
\end{equation}
The elements of ${\bf d}^\ell$ are termed {\em details} at level $\ell$.
The data ${\bf f}^j$ can be obtained exactly from the data of the pyramid
by the {\em reconstruction} steps
\begin{equation}
\label{reconstruction}
{\bf f}^\ell=\Sd{\bf f}^{\ell-1}+{\bf d}^\ell,\quad  \ell=1,2,\ldots,j.
\end{equation}
Note that every second element of ${\bf d}^\ell$ vanishes, since $\Sd$ is
interpolatory. Thus this method yields  details  satisfying
\begin{equation}\label{eq:dwn2_detail}
{\bf d}^\ell_{2k} =0,\quad k\in\Z.
\end{equation}

Multiscale transforms based on  non-interpolatory subdivision operators were studied before,
with the decimation from level $j$ 
to level
$j-1$,  defined by ``reverse subdivision". In \cite{Samavati}, ${\bf f}^{j-1} = \Dm {\bf f}^j$,  where the decimation operator $\Dm$ is related to the subdivision operator $\Sd$ by a least squares fit of $\Sd{\bf f}^{j-1}$ to ${\bf f}^j$. This method is improved in \cite{Sadegh}, by minimizing a functional consisting of two terms;
one is $\| \Sd(\Dm{\bf f}^j)-{\bf f}^j\|_2$ and the other is a roughness measure
of $ {\bf f}^{j-1}=(\Dm{\bf f}^j) $.
It is clear that in this approach the details do not satisfy \eqref{eq:dwn2_detail}.  Specific ways to reverse Chaikin  scheme and  Catmull-Clark scheme are presented in \cite{Chaikin} and \cite{CC}, respectively.

Here we propose to reverse only the even elements of ${\bf f}^j$,
as in the case of interpolatory subdivision operators, namely for a given
subdivision operator $\Sd$ to use a decimation operator $\Dm$ such that
\begin{equation}
\label{evenreversible}
 (\Sd(\Dm{\bf f}^j))_{2k}=f^{(j)}_{2k},\ \  k\in\Z.
\end{equation}
This is achieved by using Wiener's lemma \cite{Greochenig2009}, which
guarantees the existence of such a decimation operator, under mild conditions
on the mask of $\Sd$.
We term such subdivision operators, for which $\Dm$ satisfying (\ref {evenreversible})
exists, even-reversible, and refer to $\Dm$ as the even-inverse of $\Sd$.

\subsection{Outline}
Section \ref{sec:pre} consists of mathematical tools and results used in the paper, such as Wiener's lemma. In Section~\ref{sec:even}, we first give a sufficient condition on the mask of a subdivision operator guaranteeing that the operator is even-reversible,  then show that pseudo-spline subdivision operators are even-reversible, and derive explicit expression of the symbols of the even-inverse operators corresponding to the  quadratic and cubic spline subdivision operators. In Section~\ref{sec:dec:stab}, we derive the decay rate of the details in the pyramid generated by the multiscale transform and analyze the stability of the transform.  Some final remarks are given in Section~\ref{sec:remarks} and some technical proofs are postponed to the Appendix.

\section{Preliminaries}
\label{sec:pre}

In this section, we introduce some necessary notation and known results related to subdivision operators and Wiener algebra, which are used in this paper.

We denote by $l(\Z)$ the space of real-valued sequences $\ga:\Z\rightarrow\R$ and by $l_0(\Z)\subset l(\Z)$ the space of sequences of finite support.
For  $p\in[1,\infty]$, $l_p(\Z)$  denotes the usual space of  $l_p$ sequences. That is,
\[
l_p(\Z) :=\left\{\ga\in l(\Z): \|\ga\|_p:=\left(\sum_{k\in\Z} |\ga_k|^p\right)^{1/p}<\infty\right\}, \quad 1\le p<\infty,
\]
and
\[
l_\infty(\Z) :=\{\ga\in l(\Z): \|\ga\|_\infty:=\sup_{k\in\Z} |\ga_k|<\infty\}.
\]
Note that the inclusion $
l_0(\Z)\subset l_1(\Z)\subset l_p(\Z)\subset l_q(\Z)\subset l_\infty(\Z)
$
holds for any $1< p< q< \infty$.

We say that $\ga\in l(\Z)$ is a \emph{mask} if $\ga\in l_1(\Z)$. A mask $\ga$ is of finite support if $\ga\in l_0(\Z)$, that is, $\ga_k = 0$ for all $|k|\ge N$ for some integer $N$.
Given a mask $\ga$, we can define the (dyadic) \emph{ upscaling rule} or \emph{subdivision operator} $\Sd_\ga: l_\infty(\Z) \rightarrow l_\infty(\Z)$ associated with $\ga$ by
\begin{equation}
\label{def:subd}
(\Sd_{\ga} c)_k :=\sum_{\ell\in\Z} \ga_{k-2\ell} c_\ell,\quad k\in\Z, \; c\in l_\infty(\Z).
\end{equation}
It is easily seen that  $\Sd_\ga$ is a bounded linear operators on $l_\infty(\Z)$ provided that $\ga\in l_1(\Z)$. The \emph{subdivision scheme} based on $\Sd_\ga$ is the repeated application of \eqref{def:subd}, generating a sequence  of sequences
$\Sd_\ga^j c$, $  j \in\N$,  for $c\in l_\infty(\Z)$. The subdivision scheme is said to be \emph{convergent} if the sequence of piecewise linear interpolants to the data $\{(k 2^{-j}, (\Sd_\ga^j c)_k), k\in\Z\}$, $j\in\N$, is uniformly convergent for any $c\in l_\infty(\Z)$. For more about subdivision schemes see e.g. \cite{Dyn}.

We next introduce convolution, downsampling, and upsampling operations. For two sequences $\ga\in l_1(\Z)$ and $c\in l_\infty(\Z)$, we defined the \emph{convolution} $\ga*c\in l_\infty(\Z)$ to be
\[
(\ga*c)_k:=\sum_{\ell\in\Z}\ga_\ell c_{k-\ell},\quad k\in\Z,
\]
and the \emph{downsampling operator}  $\downarrow2$ applied to a sequence $c$  to be
\[
(c\downarrow 2)_k:= c_{2k},\quad k\in\Z,
\]
as well as the \emph{upsampling operator} $\uparrow2$:
\[
(c\uparrow 2)_k :=
\begin{cases}
c_{k/2} & k \mbox{ even}\\
0 & \mbox{otherwise}
\end{cases},\quad \quad k\in\Z.
\]
The subdivision in \eqref{def:subd} can be restated as
$\Sd_\ga c = \ga * (c\uparrow 2)$.

For a sequence $c\in l(\Z)$, we define its \emph{symbol} to be
\[
c(z) := \sum_{k\in\Z} c_k z^k,\quad z\in\C.
\]
Its even part $c_{ev}$ to be $(c_{ev})_k = c_{2k} = (c\downarrow 2)_k$, $k\in\Z$, and  its odd part $c_{od}$ to be $(c_{od})_k=c_{2k+1}$, $k\in\Z$. The symbols of the even and odd parts   can be determined by
\[
c_{ev}(z^2) = \frac{c(z)+c(-z)}{2}\quad\mbox{and}\quad
c_{od}(z^2) = \frac{c(z)-c(-z)}{2z}.
\]
In terms of symbolic computation, it is easily shown that
\[
\begin{aligned}
\; &[c\downarrow 2](z)  =c_{ev}(z), & [c\uparrow2](z)&=c(z^2), \\
&c(z)  =c_{ev}(z^2)+z\, c_{od}(z^2),&  [\ga*c](z) &= \ga(z)c(z).
\end{aligned}
\] Moreover, we have
$[\Sd_\ga c](z) = \ga(z)c(z^2)$.

In this paper, we  investigate the following multiscale transform based on a subdivision operator $\Sd_\ga$:
\begin{equation}
\label{def:NewSubd:forward}
c^{(j-1)} = \Dm_\gga c^{(j)} :=\gga*(c^{(j)}\downarrow 2),\quad d^{(j)} := c^{(j)}-\Sd_\ga c^{(j-1)},
\end{equation}
where $\Dm_\gamma$ is a \emph{decimation operator} associated with a mask $\gga$ to be determined.
Iterating \eqref{def:NewSubd:forward}  yields a pyramid consisting of the data $\{c^{(0)}; d^{(1)},\ldots, d^{(j)}\}$, where $c^{(0)}$ is the coarse approximation coefficients and $d^{(\ell)}$ are the detail coefficients at level $\ell=1,\ldots, j$. The reconstruction (backward transform) from $\{c^{(0)}; d^{(1)},\ldots, d^{(j)}\}$ is straightforward:
\begin{equation}
\label{def:reconstruction}
c^{(\ell)} = \Sd_\ga c^{(\ell-1)}+ d^{(\ell)}, \quad \ell=1,\ldots, j,
\end{equation}
and has the perfect reconstruction property for any pair $(\gga,\ga)$ of masks.
However, the detail coefficients $d^{(\ell)}, \ell=1,\ldots, j$ do not necessarily satisfy \eqref{eq:dwn2_detail}.

Given a subdivision operator $\Sd_\ga$ associated with a finitely supported mask $\ga\in l_0(\Z)$, we  investigate the decimation  operator $\Dm_\gga$ associated with a mask $\gga$ so that $d^{(\ell)}$ satisfies $d^{(\ell)}_{2k} =0$, $k\in\Z$, for $\ell=1,\ldots, j$; as in the case when $\Sd_\ga$ is an interpolatory subdivision operator and $\Dm_\gga c = c\downarrow 2$. Thus, we are seeking a mask $\gga$ such that
\begin{equation}
\label{cond:dwn2}
[(\Id- \Sd_\ga\Dm_\gga) c]\downarrow 2 =0
\end{equation}
for any sequence $c\in l_\infty(\Z)$, where $\Id$ is the identity operator. In such a case, the detail coefficients can be downsampled by a factor of $2$ without loss of  information, which is a   desirable property in data compression.

To solve $\gga$ from \eqref{cond:dwn2}, we use  Wiener's lemma  \cite{Greochenig2009}.
 First, we introduce Wiener algebra.

Let $\T:=\{z\in\C: |z|=1\}$.
The \emph{Wiener algebra} $W(\T)$ consists of all complex-valued functions $f$ on $[-\pi,\pi]$ such that $f$ has  absolutely convergent Fourier series. That is,
\[
W(\T):=\{f\in C([-\pi,\pi]):
\|f\|:=\sum_{n\in\Z} |\wh f(n)|<\infty
\},
\]
where $\wh f(n):=\frac{1}{2\pi}\int_{-\pi}^\pi f(x)e^{-inx}dx$ is the $n$th Fourier coefficient of $f$. The Wiener algebra $W(\T)$ is closed under pointwise multiplication of functions. It is easy to show that
\[
\|fg\| \le \|f\| \cdot \|g\| \quad \forall f,g\in W(\T).
\] Thus the Wiener algebra is a commutative unitary Banach algebra. The Wiener algebra $W(\T)$ is isomorphic to the Banach algebra $l_1(\Z)$ with the isomorphism given by the Fourier transform: $ f \mapsto \{\wh f(n)\}_{n\in\Z}$.

\begin{theorem}[Wiener's Lemma] If $f\in W(\T)$ and $f(x)\neq 0$ for all $x\in[-\pi,\pi]$, then $1/f\in W(\T)$.
\end{theorem}

Consequently, for $\ga\in l_0(\Z)$,
if $\ga(z)\neq 0$ for all $z\in\T$, then $\ga$ has an inverse $\ga^{-1}=:\gga\in l_1(\Z)$ determined by $\gga(z)=1/\ga(z)$, $z\in\T$.

\medskip

Any mask $\ga$ in the Wiener algebra defines a \emph{bi-infinite Toeplitz matrix}  of the form   $A_\ga = (\ga_{j-k})_{j,k\in\Z}$. Then, $A_\ga$ has an inverse if and only if $\ga$ has an inverse. In this case, we have that the inverse $(A_\ga)^{-1}$ of $A_\ga$ satisfies $(A_\ga)^{-1} = A_{\ga^{-1}}$.  A Toeplitz matrix $A_\ga$  is also a linear  operator on $l_2(\Z)$. One can show that the $l_2$-operator norm of $A_\ga$ is given by \cite[Section 2.2]{book1}
\begin{equation}
\label{eq:2-norm}
\|A_\ga\|_2:=\sup_{\|c\|_2=1} \|A_\ga c\|_2 = \sup_{z\in\T} |\ga(z)|.
\end{equation}

 We say that a finitely supported mask $\ga\in l_0(\Z)$ is \emph{$s$-banded} if $\ga_{k} = 0$ for all $|k|>s$. The following theorem gives  the decay of the inverse of a banded Hermitian Toeplitz matrix.

\begin{theorem}[Theorem 2.1 in \cite{Strohmer2002}]
\label{thm:A2}
Let $A$ be a bi-infinite Toeplitz matrix acting on $l_2(\Z)$ and assume $A$ to be Hermitian, positive definite, and s-banded (i.e. $A_{k,\ell} = 0$ if $|k-\ell|>s$). Set $\kappa =\|A\|_2\: \|A^{-1}\|_2$,
$\;q=(\sqrt{\kappa}-1)/(\sqrt{\kappa}+1)$, and $\lambda = q^{1/s}$. Then,
\[
(A^{-1})_{k,\ell} \le K\lambda^{|k-\ell|},\quad k, \ell\in\Z,
\]
where  $K=\|A^{-1}\|_2\max\left\{1,\frac{(1+\sqrt{\kappa})^2}{2\kappa}\right\}$.
 \end{theorem}

A direct consequence of the above results gives the decay property of the inverse $\gga$ of an invertible mask $\ga\in l_0(\Z)$, which shows that  the coefficients  of $\gga$ have exponential decay. For more general results on the decay property of the inverse of a mask, we refer to \cite{Greochenig2009}.
\begin{corollary}
\label{cor:1-norm-by-2-norm}
Let $\ga\in l_0(\Z)$ be  $s$-banded with $\ga(z)>0$ for all $z\in\T$.   Set $\kappa =\frac{\sup_{z\in\T}|\ga(z)|}{\inf_{z\in\T}|\ga(z)|}$, $q=(\sqrt{\kappa}-1)/(\sqrt{\kappa}+1)$, and $\lambda = q^{1/s}$. Then $\gga=\ga^{-1}$ exists and
\[
|\gga_\ell|\le K\lambda^{|\ell|},\quad  \ell\in\Z,
\]
where   $K=\frac{1}{\inf_{z\in\T}|\ga(z)|}\max\left\{1,\frac{(1+\sqrt{\kappa})^2}{2\kappa}\right\}$.
\end{corollary}
\begin{proof}
Since $\ga(z)>0$ for all $z\in\T$, by Wiener's lemma, $\gga = \ga^{-1}$ exists. Now by Fej\'{e}r-Riesz Lemma \cite{Daub:book}, there exists a real-valued mask $\gb \in l_0(\Z)$ such that $\ga(z) = \gb(z)\gb(1/z)$ for all $z\in\T$. Then one can easily show that $A_\ga = A_{\gb}(A_{\gb})^\tp$. Hence, by  $\ga(z)>0$ for all $z\in\T$, we conclude that $A_\ga$ is symmetric positive definite. Now by Theorem~\ref{thm:A2} and the fact that $\|A_\ga\|_2=\sup_{z\in\T}|\ga(z)|$ and $\|A_{\ga}^{-1}\|_2=\|A_{\ga^{-1}}\|_2=\inf_{z\in\T}|\ga(z)|$, we conclude the result.
\end{proof}


\section{Even-reversible subdivision}
\label{sec:even}

In this section, we give a sufficient condition on the mask of a subdivision operator   for the existence of the  linear multiscale transform defined in \eqref{def:NewSubd:forward} that satisfies  property   \eqref{cond:dwn2}. We call such  a   subdivision operator \emph{even-reversible}.  We show that for a large class of even-reversible subdivision operators,  the elements of the mask of the  decimation operator in \eqref{cond:dwn2}  decay exponentially, and prove that most of the  subdivision operators  in the literature are in that class.
\subsection{Existence of the mutliscale transform}
\label{subsec:existence}

We first deduce a sufficient condition on the mask of a subdivision operator for the existence of linear multiscale transforms  satisfying \eqref{cond:dwn2}.
\begin{theorem}
\label{thm:dwn2} Let $\Sd_\ga$ be a subdivision operator  with a mask $\ga$ defined as in \eqref{def:subd}, and let $\Dm_\gga$ be a decimation operator  with a mask $\gga$ defined as in
\eqref{def:NewSubd:forward}.
Then
\eqref{cond:dwn2} holds, i.e., $d := (\Id-\Sd_\ga\Dm_\gga)c$ satisfies $(d\downarrow2)=0$ for any $c\in l_\infty(\Z)$,   if and only if $\gga$ is the inverse of $\ga_{ev}$.
\end{theorem}
\begin{proof} Note that \eqref{cond:dwn2} is equivalent to
\[
\frac{c(z)+c(-z)}{2} - \frac{[\Sd_\ga\Dm_\gga c](z)+[\Sd_\ga\Dm_\gga c](-z)}{2} = 0;
\]
that is,
\[
c_{ev}(z^2) = [\Sd_\ga \Dm_\gga c]_{ev}(z^2)
=\ga_{ev}(z^2)[\Dm_\gga c](z^2)
 = \ga_{ev}(z^2)\gga(z^2) c_{ev}(z^2), \quad z\in\T.
\]
Consequently, \eqref{cond:dwn2} is equivalent to $\ga_{ev}(z)\gga(z)=1$ for all $z\in\T$, i.e., $\gga$ is the inverse of the even part $\ga_{ev}$ of $\ga$.
\end{proof}

Now using  Wiener's lemma, we get
\begin{corollary}
\label{cor:existence}
For any mask $\alpha\in l_0(Z)$ satisfying $\alpha_{ev}(z)\neq 0$ for all $z\in\T$, \eqref{cond:dwn2} holds with $\gga\in l_1(Z)$  determined  by $\gga(z)=\frac{1}{\alpha_{ev}(z)}$, $z\in\T$.
\end{corollary}

We call a subdivision operator $\Sd_\alpha$  with  a mask $\alpha$ such that $\alpha_{ev}(z)\neq0$ for all $z\in\T$   \emph{even-reversible}, and we call the decimation operator $\Dm_\gamma$ satisfying \eqref{cond:dwn2}  the \emph{even-inverse} of $\Sd_\ga$.
It is shown in the next subsection that for a large class of finitely supported masks, the condition $\alpha_{ev}(z)\neq0$ for all $z\in\T$ does hold.

In view of Theorem~\ref{thm:dwn2},  the multiscale transform \eqref{def:NewSubd:forward} becomes
\begin{equation}
\label{def:NewSubd:forward2}
\begin{cases}
c^{(\ell-1)}(z)& ={\ga_{ev}^{-1}(z)}{c_{ev}^{(\ell)}(z)}
\vspace{0.05in}
\\
 d_{od}^{(\ell)}(z) &=c_{od}^{(\ell)}(z)-\ga_{od}(z){\ga_{ev}^{-1}(z)}{c_{ev}^{(\ell)}(z)}
  \vspace{0.05in}
 \\
 d_{ev}^{\ell}(z)&\equiv0
 \end{cases},\quad \ell = j,\ldots,1.
\end{equation}
We term \eqref{def:NewSubd:forward2}  \emph{multiscale transform based on an even-reversible subdivision} (\emph{MTER}).

In terms of matrix computation, the MTER  can be written as
\[
\begin{cases}
c^{(\ell-1)} &=  A_{\ga_{ev}^{-1}}c_{ev}^{(\ell)}
\vspace{0.05in} \\
d_{od}^{(\ell)} &= c_{od}^{(\ell)} - A_{\ga_{od}} A_{\ga_{ev}^{-1}} c_{ev}^{(\ell)}
\vspace{0.05in}
\\
d_{ev}^{(\ell)} &\equiv0
\end{cases},\quad \ell = j,\ldots, 1.
\]
In particular, when $\ga$ is \emph{interpolatory},
\begin{equation}
\label{eq:intpSymb}
\ga_{ev}(z)\equiv1\mbox{~and~}\gga(z) = \ga_{ev}^{-1}(z)\equiv1,
\end{equation}
and the MTER is reduced to an interpolatory pyramid:
\[
\begin{cases}
c^{(\ell-1)} &= (c^{(\ell)}\downarrow 2)
\vspace{0.05in}\\
d_{od}^{(\ell)} &= c_{od}^{(\ell)} - \ga_{od} * ( c^{(\ell)}\downarrow 2)
\vspace{0.05in}\\
d_{ev}^{(\ell)} &\equiv0
\end{cases},\quad \ell = j,\ldots, 1.
\]

\subsection{The even-inverse of the subdivision operator}
In case $\alpha_{ev}>0$ for all $z\in\T$ and $\alpha_{ev}$ is $s$-banded, then by Corollary~\ref{cor:1-norm-by-2-norm},  $\gamma=\alpha_{ev}^{-1}$ exists and the elements of the mask $\gamma$ decay exponentially, which is an important feature for the computation of the decimation operation  $\Dm_\gga c$. More precisely
\[
|\gga_\ell|\le K\lambda^{|\ell|}\qquad  \forall\ell\in\Z,
\]
where $K$ and $\lambda$ are defined as in Corollary~\ref{cor:1-norm-by-2-norm} with $\alpha$ there replaced by $\alpha_{ev}$.

\subsection{Examples}
\label{subsec:examples}
In this subsection, we provide some examples of commonly used subdivision operators with finitely supported masks. We  show that for a large class of masks, the corresponding subdivision operators are even-reversible.

First, we give two examples of spline subdivision operators of low order for which we can compute explicitly the even-inverse. We omit the case of the linear spline subdivision operator because it is interpolatory.


\begin{example}[Quadratic spline]
\label{ex:spline:order3}
{\rm Consider the mask $\ga$ for the (centered) B-spline of order 3:
\[
\ga(z) = \frac{z^{-1}(1+z)^3}{2^2}.
\]
That is $\ga=\{\ga_{-1},\ga_{0},\ga_1,\ga_2\} = \frac14\{1,3,3,1\}_{[-1,2]}$ (i.e., $\supp (\ga) =[-1,2]\cap\Z$ with  $\ga_{-1}=1/4$, $\ga_{0}=\ga_1=3/4$, and $\ga_2=1/4$). Then,
\[
\ga_{ev}(z^2) = \frac12(\ga(z)+\ga(-z)) = \frac14(3+z^2),
\]
or in terms of sequence
$\ga_{ev} = \frac14\{3,1\}_{[0,1]}$. We have $|\ga_{ev}(z)|=|\frac14(3+z)|\ge\frac12$ for all $z\in\T$.
Consequently, the inverse  $\gga$ of $\ga_{ev}$ exists and is given by
\[
\gga(z) = \frac{1}{\ga_{ev}(z)} = \frac{4}{3+z}=\frac43\left(1-\frac13z+\frac19z^2+\cdots\right)=\frac43\sum_{k=0}^\infty\left(-\frac13\right)^k z^k,
\]
which is the symbol of the even-inverse operator $\Dm_\gamma$.
Thus, the quadratic spline subdivision operator is even-reversible.  Direct computations show that $\|\gga\|_1 =2$,  $\|A_\gga\|_2=2$, and $\|\gga\|_\infty=4/3$.
}
\end{example}

\begin{example}[Cubic spline]
\label{ex:spline:order4}
{\rm
Consider the mask $\ga$ for the (centered) B-spline of order 4:
\[
\ga(z) = \frac{z^{-2}(1+z)^4}{2^3}.
\]
That is $\ga=\frac18\{1,4,6,4,1\}_{[-2,2]}$. Then,
\[
\ga_{ev}(z^2) = \frac12(\ga(z)+\ga(-z)) = \frac18(z^{-2}+6+z^2),
\] or in terms of sequence
$\ga_{ev} = \frac18\{1,6,1\}_{[-1,1]}$. Note that $\ga_{ev}(z)=\frac18(z^{-1}+6+z)\ge\frac12$ for all $z\in\T$.
Consequently, the inverse  $\gga$ of $\ga_{ev}$ exists.  Hence the cubic spline subdivision operator is even-reversible, and
\[
\begin{aligned}
\gga(z)
&= \frac{1}{\ga_{ev}(z)} = \frac{8}{z^{-1}+6+z}
\\&= \frac43\times\frac{1}{1+\frac{z^{-1}+z}{6}}=\frac43\sum_{n=0}^\infty\left(\frac{-1}{6}\right)^n(z^{-1}+z)^n\\
& = \frac43\left(\sum_{n=0}^\infty6^{-2n}(z^{-1}+z)^{2n}-\sum_{n=0}^\infty 6^{-2n-1}(z^{-1}+z)^{2n+1}\right).\\
\end{aligned}
\]
By \eqref{eq:2-norm}, we have $\|A_\gga\|_2=2$.
We succeeded to obtain explicit expression of $\gga(z)$ as
\begin{equation}\label{eq:cubic:inverse}
\begin{aligned}
\gga(z) & =\sqrt{2}+\sqrt{2}\sum_{k=1}^\infty \left(\frac{-1
}{3+2\sqrt{2}}\right)^k\times\big(z^{k}+z^{-k}\big).
\end{aligned}
\end{equation}
With this explicit form of $\gga(z)$, the decimation operation in the MTER can be implemented. Moreover, from \eqref{eq:cubic:inverse}, we get   $\|\gga\|_1 = 2$,  $\|\gga\|_\infty=\sqrt{2}$,  and the exponential decay of the elements of $\gga$. The proof of \eqref{eq:cubic:inverse} is not straightforward and is given in the Appendix.
}
\end{example}

\medskip

From the above examples, one can expect that spline subdivision operators are even-reversible. In fact, it holds more generally. The class of spline subdivision operators can be regarded as a special subclass of a larger class called \emph{pseudo-spline} subdivision operators, which we introduce next.

Let $n,\nu\in\N\cup\{0\}$ satisfying $0\le \nu\le \lfloor n/2\rfloor-1$. Define
\begin{equation}
\label{def:pseudo-spline}
\begin{aligned}
\ga^{n,\nu}(z) &= \frac{z^{-\lfloor n/2\rfloor}(1+z)^{n}}{2^{n-1}}\sum_{j=0}^\nu{n/2+j-1\choose j} \left(\frac12-\frac{z+z^{-1}}{4}\right)^j.
\end{aligned}
\end{equation}
When $\nu=0$, the mask $\ga^{n,0}=\ga^n$ belongs to the family of masks of spline subdivision operators. When $n=2k$ and $\nu=k-1$, the mask $\ga^{2k,k-1}$ is the  Deslauries-Dubuc's interpolatory mask \cite{DD}.   For  $n=2k$  and $0\le \nu\le k-1$, the masks $\ga^{2k,\nu}$ are the masks of the  primal pseudo-splines (pseudo-splines of type II);
 while for $n=2k+1$, the masks $\ga^{2k+1,\nu}$, $0\le\nu\le k-1$ are  the masks of the dual pseudo-splines. For more about pseudo-splines, see \cite{DauHanRonShen, DongShen, DongDynHormann,DHSS}  and references therein.

The next theorem states that pseudo-spline subdivision operators are even-reversible, which implies  that all spline subdivision operators are even-reversible.

\begin{theorem}
\label{thm:pseudo-spline:inverse}
Let $n,\nu\in\N\cup\{0\}$ satisfying $0\le \nu\le \lfloor n/2\rfloor-1$, and  $\ga^{n,\nu}$  be the pseudo-spline mask defined as in \eqref{def:pseudo-spline}.
Then the following holds:
\begin{itemize}

\item[\rm{(1)}] $\ga^{n,\nu}_{ev}(1)=1$ and $\|A_{\ga^{n,\nu}_{ev}}\|_2 =\max_{z\in\T}|\ga_{ev}^{n,\nu}(z)| = 1$.

\item[\rm{(2)}]  $\gga=(\ga_{ev}^{n,\nu})^{-1}$ exists with $\gga(1) = 1$.

\item[\rm{(3)}]  The $l_2$-norm of $A_\gga$ is given by

\[
\|A_\gga\|_2  = \frac{2^{\lfloor \frac{n-1}{2}\rfloor+\nu}}{\sum_{j=0}^\nu{n/2+\nu\choose j}}.
\]

\item[\rm{(4)}] The elements of $\gga$ decay exponentially.
\end{itemize}
\end{theorem}

\bigskip
The proof of Theorem~\ref{thm:pseudo-spline:inverse} is postponed to the Appendix. We
remark that
\begin{itemize}
\item[\rm(i)]
when $\nu=0$, we have $\|A_\gga\|_2 = 2^{\lfloor (n-1)/2\rfloor}$;

\item[\rm(ii)]
when $n=2k$ and $\nu=k-1$, we have $\gga(z) \equiv 1$.


\end{itemize}

\section{Decay and stability}
\label{sec:dec:stab}
 We now turn to the study of the decay property of the pyramid  sequences and the stability of the transforms.

\subsection{Decay}
\label{subsec:decay}

Let $\Delta: l(\Z)\rightarrow l(Z)$ denote the \emph{difference operator}: $(\Delta c)_k = c_{k+1}-c_k$ for $c\in l(\Z)$. Then, $\Delta$  commutes with convolution operators. Indeed, by
\[
[\Delta c](z) = (z^{-1}-1) c(z)\mbox{~and~}[\gga*c](z)=\gga(z)c(z),
\]
we have
\[
[\Delta (\gga*c)](z) =(z^{-1}-1)(\gga(z)c(z)) = \gga(z)[(z^{-1}-1)c(z)] = [\gga* (\Delta c)](z).
\]

Suppose the data sequence $c^{(j)} = f\big|_{2^{-j}\Z}$, where $f$ is a  function in  $C^1(\R)$ with bounded first derivative. Then,
\[
[\Delta c^{(j)}]_k=f(2^{-j}(k+1))-f(2^{-j}k) =f'(\xi)\cdot 2^{-j}
\]
for some $\xi\in(2^{-j}k,2^{-j}(k+1))$. Thus
$\|\Delta c^{(j)}\|_\infty \le K 2^{-j}$
with $K=\|f'\|_\infty$.


Applying the scheme in \eqref{def:NewSubd:forward}, we have a pyramid of data consisting of approximation coefficient sequences $c^{(\ell)}$ and detail coefficient sequences $d^{(\ell)}$. The following two results concern the decay property of $\Delta c^{(\ell)}$ and $d^{(\ell)}$.
\begin{theorem}[Difference of approximation coefficients]
\label{thm:decay:c}
Let $c^{(j)}$ be such that $\|\Delta c^{(j)}\|_\infty\le K2^{-j}$. Let $c^{(\ell)} := \gga * (c^{(\ell+1)}\downarrow 2)$ with $\gga\in l_1(\Z)$ and $0\le\ell\le j-1$. Then,
\begin{equation}
\label{decay:delta_c}
\|\Delta c^{(\ell )}\|_\infty \le K\: \|\gga\|_1^j\cdot (2\|\gga\|)^{-\ell},\quad 0\le \ell\le j-1.
\end{equation}
\end{theorem}

\begin{proof}
For $c_{ev} = c\downarrow 2$, we have
\[
(\Delta c_{ev})_k= c_{2k+2}-c_{2k}= (c_{2k+2}-c_{2k+1})+(c_{2k+1}-c_{2k})=(\Delta c)_{2k+1}+(\Delta c)_{2k},\quad k\in\Z,
\]
and hence, $\|\Delta c_{ev}\|_\infty \le 2\|\Delta c\|_\infty$. Since $\Delta$ commutes with  convolution operators, we have
\[
\begin{aligned}
\|\Delta c^{(j-1)}\|_\infty
&  =  \|\Delta(\gga* (c^{(j)}\downarrow 2))\|_\infty
\\&=\| \gga* \Delta(c^{(j)}\downarrow 2))\|_\infty
\\&\le \|\gga\|_1\cdot \|  \Delta(c^{(j)}\downarrow 2)\|_\infty
\\&\le\|\gga\|_1\cdot 2\|\Delta c^{(j)}\|_\infty.
\end{aligned}
\]
Iterating the above inequality starting with $\|\Delta c^{(j)}\|_\infty\le K2^{-j}$, we conclude \eqref{decay:delta_c}.
\end{proof}

Under some very mild conditions  on the masks $\ga, \gga$, we can show that the detail coefficients have the same decay property as the differences of the approximation coefficients.

\begin{theorem}[Detail  coefficients]
\label{thm:decay:d}
Let $c^{(j)}\in l_\infty(\Z)$ be a sequence. Define $c^{(\ell-1)}$ and $d^{(\ell)}$ to be
\[
c^{(\ell-1)} :=\gga*(c^{(\ell)}\downarrow2), \quad
d^{(\ell)} := c^{(\ell)}- \Sd_\ga c^{(\ell-1)},
\quad 1\le \ell \le j,
\]
where $\ga,\gga\in l_1(\Z)$ are masks satisfying
\begin{equation}\label{cond:filter:normalization}
\sum_{k} \ga_{2k} = \sum_k \ga_{2k+1}=\sum_k\gga_k=1
\end{equation}
and
\begin{equation}
\label{cond:filter:decay}
\sum_{k\in\Z}|\ga_k| \;|k|=K_\ga<\infty,\quad
2\sum_{k\in\Z}|\gga_{k}| \;|k|=K_\gga < \infty.
\end{equation}
 Then,
\begin{equation}
\label{decay:detail}
\|d^{(\ell)}\|_\infty\le K_{\ga,\gga} \|\Delta c^{(\ell)}\|_\infty,\quad 1\le \ell\le j
\vspace{0.06in}
\end{equation}
with $K_{\ga,\gga} := K_\gga\|\ga\|_1+K_\ga \|\gga\|_1$. In particular, if  $\;\|\Delta c^{(j)}\|_\infty\le K \cdot  2^{-j}$ for some $K$ independent of $j$, then,
\begin{equation}
\label{decay:detail2}
\|d^{(\ell)}\|_\infty\le \left(K \cdot K_{\ga,\gga}\cdot \|\gga\|_1^{j} \right)\cdot (2\|\gga\|)^{-\ell}, \quad 1\le \ell \le j.
\end{equation}
\end{theorem}

\begin{proof}
Let $\eta =\{\eta_k\}_{k\in\Z}:=c^{(\ell-1)}=\gga*(c^{(\ell)}\downarrow 2)$. Then, $\eta_k = \sum_{s} \gga_{k-s} c_{2s}^{(\ell)}$. By \eqref{cond:filter:normalization}, we have
\[
\begin{aligned}
d^{(\ell)}_k &=c_k^{(\ell)} - [\Sd_\ga \eta]_k
=c_k^{(\ell)}-\sum_{s} \ga_{k-2s} \eta_s
= \sum_{s} \ga_{k-2s}\Big[\sum_{n}\gga_{s-n}(c_k^{(\ell)}-c_{2n}^{(\ell)})\Big].
\end{aligned}
\]
Consequently,
\[
\begin{aligned}
\|d^{(\ell)}\|_\infty
&\le \sum_{s}|\ga_{k-2s}|\Big[\sum_{n} |\gga_{s-n}||2n-k|\Big]\|\Delta c^{(\ell)}\|_\infty
\\&
\le  \sum_{s}|\ga_{k-2s}|\Big[\sum_{n} |\gga_{s-n}|(|2n-2s|+|k-2s|)\Big]\|\Delta c^{(\ell)}\|_\infty
\\&\le \sum_{s}|\ga_{k-2s}|\Big[K_\gga+|k-2s|\|\gga\|_1\Big]\|\Delta c^{(\ell)}\|_\infty
\\&=(K_\gga\|\ga\|_1+K_\ga \|\gga\|_1)\|\Delta c^{(\ell)}\|_\infty
\\&=K_{\ga,\gga}\|\Delta c^{(\ell)}\|_\infty
\\&\le  K \cdot K_{\ga,\gga}\cdot \|\gga\|_1^{j} \cdot (2\|\gga\|)^{-\ell},
\end{aligned}
\]
where the last inequality follows from \eqref{decay:delta_c}.
\end{proof}

\medskip
Theorems~\ref{thm:decay:c} and \ref{thm:decay:d} imply

\begin{corollary}
\label{cor:56}
Let $\ga\in l_0(\Z)$ be a mask of a convergent subdivision scheme satisfying $\ga_{ev}(z)>0$ for $z\in\T$, and let $c^{(j)}$ be a sequence satisfying $\|\Delta c^{(j)}\|_\infty\le K 2^{-j}$. Then the pyramid generated by the MTER in \eqref{def:NewSubd:forward2} satisfies \eqref{decay:delta_c} and \eqref{decay:detail2}.
\end{corollary}

We remark that in the interpolatory case, \eqref{eq:intpSymb} holds. Therefore, inequalities \eqref{decay:delta_c} and \eqref{decay:detail2} depend on the level $\ell$ and are independent of $j$. For non-interpolatory subdivision operators and for a fixed $j$, the decay of the details is faster but the constant is bigger since $\|\gga\|_1>1$.


\subsection{Stability of reconstruction}
\label{subsec:stab:rec}
In this subsection,
we show that   the reconstruction is stable.
\begin{theorem}
Suppose $\sup_{j\in\N}\|\Sd_\ga^j\|_\infty\le K$ for some constant $K>0$. Then the reconstructed data $c^{(j)}$ at level $j$ from coarse data $c^{(0)}$ and details $d^{1},\ldots, d^{(j)}$ via \eqref{def:reconstruction} is stable; that is,
\[
\|c^{(j)}-\wt{c}^{(j)}\|_\infty\le K\left(\|c^{(0)}-\wt{c}^{(0)}\|_\infty+\sum_{\ell=1}^{j}\|d^{(\ell)}-\wt{d}^{(\ell)}\|_\infty\right),
\]
where $\wt{c}^{(j)}$ is reconstructed from the data $\wt{c}^{(0)}$, $\wt{d}^{1}$, $\ldots$, $\wt{d}^{(j)}$ via \eqref{def:reconstruction}.
\end{theorem}
\begin{proof}
From $c^{(\ell)} = \Sd_\ga c^{(\ell-1)}+d^{(\ell)}$, we have
\[
\begin{aligned}
\|c^{(j)}-\wt{c}^{(j)}\|_\infty
&=\|\Sd_\ga (c^{(j-1)}-\wt{c}^{(j-1)})+(d^{(j)}-\wt{d}^{(j)})\|_\infty\\
&=\|\Sd_\ga^2 (c^{(j-2)}-\wt{c}^{(j-2)})+\Sd_\ga(d^{(j-1)}-\wt{d}^{(j-1)})+(d^{(j)}-\wt{d}^{(j)}) \|_\infty\\
& \;\; \vdots\\
&=\Big\|\Sd_\ga^j (c^{(0)}-\wt{c}^{(0)})+\sum_{\ell=1}^j\Sd_\ga^{j-\ell}(d^{(\ell)}-\wt{d}^{(\ell)}) \Big\|_\infty\\
&\le \|\Sd_\ga^j\|_\infty\|c^{(0)}-\wt{c}^{(0)}\|_\infty+\sum_{\ell=1}^j\|\Sd_\ga^{j-\ell}\|_\infty\|d^{(\ell)}-\wt{d}^{(\ell)}\|_\infty\\
&\le K\left(\|c^{(0)}-\wt{c}^{(0)}\|_\infty+\sum_{\ell=1}^{j}\|d^{(\ell)}-\wt{d}^{(\ell)}\|_\infty\right).
\end{aligned}
\]
\end{proof}

It is well known and easy to see from the uniform boundedness principle that the condition 
\[
\sup_{j\in\N}\|\Sd_\ga^j\|_\infty\le K
\] 
holds whenever the subdivision scheme based on $\Sd_\ga$ is convergent.

\subsection{Stability of  decomposition}
\label{subsec:stab:dec}
We now  show  stability of the decomposition  of the MTER in \eqref{def:NewSubd:forward2} for fixed $j$.

\begin{theorem}
\label{thm:dec:stab}
Suppose $c^{(j)}, \wt{c}^{(j)}\in l_p(\Z)$ for $p\in[1,\infty]$. Let the the pyramid data  $\{c^{(0)}; d^{(1)},\ldots, d^{(j)}\}$
and $\{\wt{c}^{(0)}; \wt{d}^{(1)},\ldots, \wt{d}^{(j)}\}$ be otained from $c^{(j)}$ and $\wt{c}^{(j)}$ respectively by the scheme \eqref{def:NewSubd:forward}. Then, we have for $\ell=1,\ldots, j$,
\begin{equation}
\label{dec:stab:lp}
\begin{aligned}
 \|c^{(0)}-\wt{c}^{(0)}\|_p &\le \|\Dm_\gga\|_p^j\cdot \|c^{(j)}-\wt{c}^{(j)}\|_p,\\
 \|d^{(\ell)}-\wt{d}^{(\ell)}\|_p & \le \left(\|\Id-\Sd_\ga \Dm_\gga\|_p \cdot \|\Dm_\gga\|_p^{j} \cdot \|c^{(j)}-\wt{c}^{(j)}\|_p\right)\cdot \|\Dm_\gga\|_p^{-\ell}.
 \end{aligned}
\end{equation}
\end{theorem}
\begin{proof}
By
\[
\|c^{(\ell-1)}-\wt{c}^{(\ell-1)}\|_p = \|\Dm_\gga(c^{(\ell)}-\wt{c}^{(\ell)})\|_p\le \|\Dm_\gga\|_p\|c^{(\ell)}-\wt{c}^{(\ell)}\|_p,
\]
we have
\[
\|c^{(0)}-\wt{c}^{(0)}\|_p \le \|\Dm_\gga\|_p^j\|c^{(j)}-\wt{c}^{(j)}\|_p.
\]
Similarly, by
\[
\|d^{(\ell)}-\wt{d}^{(\ell)}\|_p  = \|(\Id-\Sd_\ga\Dm_\gga)(c^{(\ell)}-\wt{c}^{(\ell)})\|_p\le \|\Id-\Sd_\ga\Dm_\gga\|_p\|c^{(\ell)}-\wt{c}^{(\ell)})\|_p,
\]
we have
\[
\|d^{(\ell)}-\wt{d}^{(\ell)}\|_p\le
\|(\Id-\Sd_\ga\Dm_\gga)\|_p\|\Dm_\gga\|_p^{j-\ell}\|c^{(j)}-\wt{c}^{(j)}\|_p.
\]
\end{proof}

Remarks:

\begin{itemize}
\item[(i)]
For $\Sd_\ga$  an interpolatory subdivision operator, \eqref{eq:intpSymb} holds. Therefore in the corresponding MTER, $\Dm_\gga c$  is simply $\Dm_\gga c=c\downarrow 2$ and  $\|\Dm_\gga\|_p\equiv1$. It follows from Theorem~\ref{thm:dec:stab} that the decomposition of MTERs based on  interpolatory subdivision operators is  stable  for all $p\in [1,\infty]$, namely, the constants in \eqref{dec:stab:lp} are independent of $j$.

\item[(ii)]
For non-interpolatory subdivision operators corresponding to convergent even-reversible subdivision schemes, the even-inverse $\Dm_\gamma$ satisfies $\|\Dm_\gga\|_p \ge 1$ since $\gga(1)=1$. In such a case, the constant $\|\Dm_\gga\|_p^j$ depends on the level $j$ of the data.
\end{itemize}


In the  following, we give bounds on  $\|\Dm_\gga\|_\infty=\|\gga\|_1$  for the family of primal pseudo-spline subdivision operators.




\begin{theorem}
\label{thm:1-norm}
Let $\ga^{2k,\nu}$ be the mask as defined in \eqref{def:pseudo-spline} with $0\le \nu\le k-1$  and $k \ge 2$, and let $\gga=(\ga^{2k,\nu}_{ev})^{-1}$. Then,
\[
\|\Dm_\gga\|_\infty=\|\gga\|_1 \le C(k,\nu),
\]
where
\begin{equation}
\label{def:C_k_nu}
C(k,\nu)=\kappa \cdot  \max\left\{1,\frac{(1+\sqrt{\kappa})^2}{2\kappa}\right\}\cdot \frac{1+\lambda}{1-\lambda}
\end{equation}
is a constant depending on $k$ and $\nu$ with
\[
\kappa = \frac{2^{k+\nu-1}}{\sum_{j=0}^{\nu} {k+{\nu} \choose j}},\;
\; \lambda = \left(\frac{\sqrt{\kappa}-1}{\sqrt{\kappa}+1}\right)^{1/s},\;\mbox{and}\; s=\lfloor (k+\nu)/2\rfloor.
\]
\end{theorem}
\bigskip

The proof of Theorem~\ref{thm:1-norm} is given in the Appendix. We remark that
\begin{itemize}
\item[(i)] For the case $\nu=k-1$, $\ga^{2k,k-1}$ corresponds to the family of Deslauries-Dubuc's interpolatory masks \cite{DD}. In such a case, the constant $C(k,k-1)$ is exact in the sense that $C(k,k-1)=1$.

\item[(ii)] For the case $\nu=0$, $\ga^{2k,0}$ corresponds to the family of spline masks of even order. In such a case, it can be shown that  $C(k,0) = O(k\cdot 2^{\frac{3k}{2}} ) $ for $k\rightarrow \infty$.
\end{itemize}

\medskip

Combining the above results, we have the following result regarding the  stability  of decomposition of the MTER based on primal pseudo-spline masks in $l_p(\Z)$ for two important cases $p=2$ and $p=\infty$.

\begin{corollary}
\label{cor:dec:stab}
Suppose $c^{(j)}, \wt{c}^{(j)}\in l_p(\Z)$ for $p\in[1,\infty]$. Let the the pyramid data  $\{c^{(0)}; d^{(1)},\ldots, d^{(j)}\}$
and $\{\wt{c}^{(0)}; \wt{d}^{(1)},\ldots, \wt{d}^{(j)}\}$ be obtained from $c^{(j)}$ and $\wt{c}^{(j)}$ respectively by the scheme \eqref{def:NewSubd:forward} with $\ga=\ga^{2k,\nu}$ for $0\le \nu\le k-2$ and $k\ge 2$. Then,
\begin{itemize}
\item[\rm{(1)}] for $p=\infty$, we have\begin{equation}
\label{dec:stab:linfty:spline}
\begin{aligned}
\|c^{(0)}-\wt{c}^{(0)}\|_\infty & \le C(k,\nu)^j \cdot\|c^{(j)}-\wt{c}^{(j)}\|_\infty,\\
\|d^{(\ell)}-\wt{d}^{(\ell)}\|_\infty & \le \left(\|\Id-\Sd_{\ga^{2k,\nu}}\Dm_\gga\|_\infty \cdot C(k,\nu)^{j}\cdot \|c^{(j)}-\wt{c}^{(j)}\|_\infty\right)\cdot C(k,\nu)^{-\ell}\\
\end{aligned}
\end{equation}
for $\ell = 1,\ldots, j$, where $C(k,\nu)$ is the constant  defined  in \eqref{def:C_k_nu}.

\item[\rm{(2)}] For $p=2$, \eqref{dec:stab:lp} holds with $\|\Dm_\gga\|_2=\|A_\gga\|_2 = \frac{2^{k+\nu-1}}{\sum_{j=0}^\nu {k+\nu\choose j}}$.
%
\end{itemize}
\end{corollary}
\begin{proof}
This is a direct consequence of Theorems~\ref{thm:dec:stab} and~\ref{thm:1-norm}.
\end{proof}

\section{Final Remarks}
\label{sec:remarks}
In this section, we give some further remarks.
\begin{enumerate}[1)]

\item Most examples and results in this paper deal with finitely supported masks for the purpose of simplicity of presentation. We point out that the MTER in \eqref{def:NewSubd:forward2} applies for any mask $\ga$ provided $\ga_{ev}(z)\neq 0$ for all $z\in\T$.

\item Using the weighted version of Wiener's lemma \cite{Greochenig2009}, one can study  classes of masks of infinite
support such as masks with polynomial or sub-exponential decay, and masks corresponding to rational symbols.
The extension of the class of masks might lead to even-inverse operators of smaller norm, and by that to the improvement of the stability  of the decomposition and the decay rate of the details as given in Theorem \ref{thm:dec:stab} and Theorem~\ref{thm:decay:d}, respectively.

\item The existence of MTER relies on  Wiener's lemma. Since a high-dimensional version of Wiener's lemma holds, our MTER in \eqref{def:NewSubd:forward2}  can be  generalized to any dimension $d\in\N$. We shall report related results elsewhere.


\end{enumerate}



\appendix
\section*{Appendix}
\label{sec:proofs}
\section{Proof of \eqref{eq:cubic:inverse}}
\begin{proof}[Proof of \eqref{eq:cubic:inverse}]
Note that
\[
\gga(z) =\frac{8}{z^{-1}+6+z}
 = \frac43\left(\sum_{n=0}^\infty6^{-2n}(z^{-1}+z)^{2n}-\sum_{n=0}^\infty 6^{-2n-1}(z^{-1}+z)^{2n+1}\right).
\]
For $\sum_{n=0}^\infty(\frac16)^{2n}(z^{-1}+z)^{2n}$, we have
\[
\begin{aligned}
\sum_{n=0}^\infty6^{-2n}(z^{-1}+z)^{2n}
&=\sum_{n=0}^\infty6^{-2n}\sum_{k=0}^{2n}{2n \choose k}z^{2(n-k)}
\\&=\sum_{n=0}^\infty6^{-2n}\sum_{k=-n}^{n}{2n \choose n- k}z^{2k}
\\&=\sum_{n=0}^{\infty}{2n \choose n}6^{-2n}+\sum_{k=1}^\infty\left[\sum_{n=k}^{\infty}{2n \choose n-k}6^{-2n}\right](z^{2k}+z^{-2k})
\\&=\sum_{n=0}^{\infty}{2n \choose n}6^{-2n}+\sum_{k=1}^\infty\left[\sum_{n=0}^{\infty}{2(n+k) \choose n}6^{-2(n+k)}\right](z^{2k}+z^{-2k}) .
\end{aligned}
\]
Similarly, for $\sum_{n=0}^\infty (\frac16)^{2n+1}(z^{-1}+z)^{2n+1}$, we have
\[
\begin{aligned}
\sum_{n=0}^\infty 6^{-2n-1}(z^{-1}+z)^{2n+1}
&=\sum_{k=0}^\infty\left[\sum_{n=0}^{\infty}{2(n+k)+1 \choose n}6^{-2(n+k)-1}\right](z^{2k+1}+z^{-2k-1}) .
\end{aligned}
\]
Define for $k=0,1,2,\ldots$,
\[
a_k:=\sum_{n=0}^{\infty}{2(n+k) \choose n}6^{-2(n+k)}\quad\mbox{and}\quad
b_k:=\sum_{n=0}^{\infty}{2(n+k)+1 \choose n}6^{-2(n+k)-1}.
\]
Then, we have
\[
\begin{aligned}
\gga(z)& = \frac43\sum_{n=0}^\infty\left(\frac{-1}{6}\right)^n(z^{-1}+z)^n \\
&=\frac43\left[ a_0+\sum_{k=1}^\infty a_k(z^{2k}+z^{-2k})-\sum_{k=0}^\infty b_k(z^{2k+1}+z^{-2k-1})\right].
\end{aligned}
\]
Using the formula ${n+1\choose m} = {n \choose m}+{n \choose m-1}$, it is easy to check that
\[
a_k = \frac{1}{6}(b_{k-1}+b_k)\quad\mbox{and}\quad b_k =\frac16(a_k+a_{k+1}).
\]
Hence
\begin{equation}
\label{eq:Recursive_Cubic_Inverse}
a_{k+1}=6b_k-a_k
\quad\mbox{and}\quad
b_k =6a_k-b_{k-1}.
\end{equation}
Next,  we  proceed to prove the following identities using the recurrence relations \eqref{eq:Recursive_Cubic_Inverse}:
\begin{equation}\label{eq:formular:cubic-inverse}
a_k = \frac{3\sqrt{2}}{4}(3-2\sqrt{2})^{2k}\quad\mbox{and}\quad
b_k = \frac{3\sqrt{2}}{4}(3-2\sqrt{2})^{2k+1},\quad k=0,1,2,\ldots.
\end{equation}
First, we compute directly $a_0$ and $b_0$. Using the identities
\begin{equation}
\label{eq:series}
(1-x)^{-t} = \sum_{n=0}^\infty{t-1+n\choose n} x^n,
\end{equation}
and
\[
{2n\choose n}  = 4^n {n-1/2\choose n},
\]
we have
\[
a_0=\sum_{n=0}^{\infty}{2n \choose n}6^{-2n} = \sum_{n=0}^{\infty}{1/2-1+n\choose n}9^{-n}=(1-1/9)^{-1/2}= \frac{3\sqrt{2}}{4}.
\]
Similarly,
\[
b_0=\sum_{n=0}^{\infty}{2n+1 \choose n}6^{-2n-1} = 3\left[\sum_{n=0}^{\infty}{1/2-1+n\choose n}9^{-n}-1\right]= \frac{3\sqrt{2}}{4}(3-2\sqrt{2}).
\]
Now, recursively using the relation of $a_k$ and $b_k$ in \eqref{eq:Recursive_Cubic_Inverse}, we have
\[
\begin{aligned}
a_{k+1} &= 6b_k-a_k = \frac{3\sqrt{2}}{4} (6(3-2\sqrt{2})^{2k+1}-(3-2\sqrt{2})^{2k})
\\&=  \frac{3\sqrt{2}}{4} (3-2\sqrt{2})^{2k}(6(3-2\sqrt{2})-1)
\\&= \frac{3\sqrt{2}}{4} (3-2\sqrt{2})^{2k}(3-2\sqrt{2})^2
\\& =\frac{3\sqrt{2}}{4} (3-2\sqrt{2})^{2(k+1)},
\end{aligned}
\]
and
\[
\begin{aligned}
b_k&= 6a_k-b_{k-1} = \frac{3\sqrt{2}}{4} (6(3-2\sqrt{2})^{2k}-(3-2\sqrt{2})^{2k-1})
\\&=  \frac{3\sqrt{2}}{4} (3-2\sqrt{2})^{2k-1}(6(3-2\sqrt{2})-1)
\\&= \frac{3\sqrt{2}}{4} (3-2\sqrt{2})^{2k-1}(3-2\sqrt{2})^2
\\& =\frac{3\sqrt{2}}{4} (3-2\sqrt{2})^{2k+1}.
\end{aligned}
\]
Therefore,  \eqref{eq:formular:cubic-inverse} holds.
In summary, we conclude that
\[
\begin{aligned}
\gga(z)  & = \frac43\left[a_0+\sum_{k=1}^\infty a_k(z^{2k}+z^{-2k})-\sum_{k=0}^\infty b_k(z^{2k+1}+z^{-2k-1})\right]\\
 & = \sqrt{2}\left[1+\sum_{k=1}^\infty (3-2\sqrt{2})^{2k}(z^{2k}+z^{-2k})-\sum_{k=0}^\infty (3-2\sqrt{2})^{2k+1}(z^{2k+1}+z^{-2k-1})\right],
 \end{aligned}
\]
which proves \eqref{eq:cubic:inverse}.

\end{proof}


\section{Proof of Theorem~\ref{thm:pseudo-spline:inverse}}

\begin{proof}[Proof of Theorem~\ref{thm:pseudo-spline:inverse}]
Let $z=e^{-i\theta}$ and $x=\sin^2(\theta/2)$. One can show that
\[
\ga^{n,\nu}(e^{-i\theta}) = 2e^{i(\lfloor n/2\rfloor -n/2)\theta}\cos^n(\theta/2) Q_{n,\nu}(\sin^2(\theta/2)),
\]
with
\[
Q_{n,\nu}(x) := \sum_{j=0}^\nu{n/2-1+j \choose j} x^j =\frac{1}{(1-x)^{n/2}}+O(x^{\nu+1}),
\]
where in the last equality we use \eqref{eq:series}.
Note that $Q_{n,\nu}(x)\ge1$ for all $x\ge0$. By  that $\ga_{ev}^{n,\nu}(z^2) = \frac12(\ga^{n,\nu}(z)+\ga^{n,\nu}(-z))$, we have
\[
\ga_{ev}^{n,\nu}(e^{-2i\theta}) = e^{i(\lfloor n/2\rfloor -n/2)\theta}\Big[(1-x)^{n/2} Q_{n,\nu}(x)
+i^{2\lfloor n/2\rfloor - n}(-1)^n x^{n/2} Q_{n,\nu}(1-x)\Big].
\]
It is easy to see that
$\ga^{n,\nu}_{ev}(1)=1$.

For $n=2k$, we have
\begin{equation}
\label{eq:2k}
\ga_{ev}^{2k,\nu}(e^{-2i\theta}) = (1-x)^{k}Q_{2k,\nu}(x)+x^{k}Q_{2k,\nu}(1-x)=:R(x)+R(1-x),
\end{equation}
where
\[
R(x) := (1-x)^kQ_{2k,\nu}(x)=(1-x)^{k}\sum_{j=0}^\nu{k-1+j \choose j}x^j.
\]
 Define $g(x):=R(x)+R(1-x)$. By using $(j+1){k+j \choose j+1}-j{k-1+j \choose j}=k{k-1+j \choose j}$, one can show that
 \[
 R'(x)=-(k+\nu){k-1+\nu \choose \nu}(1-x)^{k-1}x^\nu.
\]
Thus,
\[
g'(x)=R'(x)-R'(1-x)=(k+\nu){k-1+\nu \choose \nu}x^\nu(1-x)^\nu\Big[x^{k-1-\nu}-(1-x)^{k-1-\nu}\Big].
\]
It is easily seen that $g'(x)\le 0$ for $x\in[0,1/2]$ and $g'(x)\ge0$ for $x\in[1/2,1]$.
Consequently,
\[
\begin{aligned}
\min_{z\in\T}|\ga_{ev}^{2k,\nu}(z)|
& = \min_{z\in\T}\ga_{ev}^{2k,\nu}(z)\\
& =\min_{x\in[0,1]}g(x) = g(1/2)=2R(1/2)\\
& =2^{1-k}\sum_{j=0}^\nu {k-1+j \choose j} 2^{-j}>0,
\end{aligned}
\]
where the last equation  can be shown to equivalent to
\begin{equation}
\label{eq:primalPositivity}
\min_{z\in\T}|\ga_{ev}^{2k,\nu}(z)|
=2^{1-k-\nu}\sum_{j=0}^\nu {k+\nu \choose j}>0,
\end{equation}
Moreover, by \eqref{eq:2-norm}, we have
 \[
 \|A_{\ga_{ev}^{2k,\nu}}\|_2=\max_{z\in\T}|\ga_{ev}^{2k,\nu}(z)| =\max_{x\in[0,1]} g(x)= g(0)=g(1)=1.
 \]

For $n=2k+1$, we have
\[
\ga_{ev}^{2k+1,\nu}(e^{-2i\theta})
=e^{-i\theta/2}\left[(1-x)^{k+1/2}Q_{2k+1,\nu}(x)+i \cdot x^{k+1/2}Q_{2k+1,\nu}(1-x)\right],
\]
from which, we have
\[
\begin{aligned}
|\ga_{ev}^{2k+1,\nu}(e^{-2i\theta})|^2
&=(1-x)^{2k+1}(Q_{2k+1,\nu}(x))^2
+x^{2k+1}(Q_{2k+1,\nu}(1-x))^2
\\&=:R(x)^2+R(1-x)^2,
\end{aligned}
\]
where
\[
R(x):=(1-x)^{k+1/2}Q_{2k+1,\nu}(x)= (1-x)^{k+1/2}\sum_{j=0}^\nu{k-1/2+j \choose j}x^j.
\]
 Define $g(x)=R(x)^2+R(1-x)^2$. Then by using $(j+1){k+1/2+j \choose j+1}-j{k-1/2+j \choose j}=(k+1/2){k-1/2+j \choose j}$, similarly, we can show that
 \[
 R'(x) = -(k+1/2+\nu){k-1/2+\nu \choose \nu}(1-x)^{k-1/2}x^\nu.
 \]
Consequently,
\[
\begin{aligned}
g'(x)&=2\big[R(x)R'(x)-R(1-x)R'(1-x)\big]
\\&=-2c x^\nu(1-x)^\nu \left[\sum_{j=0}^\nu {k-1/2+j\choose j}(1-x)^jx^j\Big[(1-x)^{2k-\nu-j}-x^{2k-\nu-j}\Big]\right],
\end{aligned}
\]
where $c=(k+1/2+\nu){k-1/2+\nu \choose \nu}$. Now it is easy to see that  each term
\[
t_j(x):={k-1/2+j\choose j}(1-x)^jx^j\big[(1-x)^{2k-\nu-j}-x^{2k-\nu-j}\big]
\]
in the above summation for $j=0,\ldots, \nu$,   satisfies
\[
t_j(x)\ge0 \mbox{ for } x\in[0,1/2]\mbox{ and  }t_j(x)\le0\mbox{ for }x\in[1/2,1].
\]
 Hence $g'(x)\le 0$ for $x\in[0,1/2]$ and $g'(x)\ge0$ for $x\in[1/2,1]$.
Consequently,
\[
\begin{aligned}
\min_{x\in[0,1]}g(x) & = g(1/2)=2R(1/2)^2
\\&=2\left(2^{-k-1/2}\sum_{j=0}^\nu {k-1/2+j \choose j} 2^{-j}\right)^2
\\&=2^{-2k-2\nu}\left(\sum_{j=0}^\nu {k+1/2+\nu \choose j}\right)^2.
\end{aligned}
\]
Therefore,
\[
\min_{z\in\T}|\ga_{ev}^{2k+1,\nu}(z)| = \min_{x\in[0,1]}\sqrt{g(x)} = 2^{-k-\nu}\sum_{j=0}^\nu {k+1/2+\nu \choose j}>0,
\]
and  by \eqref{eq:2-norm}, we have
\[
\|A_{\ga_{ev}^{2k+1,\nu}}\|_2=\max_{z\in\T}|\ga_{ev}^{2k+1,\nu}(z)| = g(0)=g(1)=1.
\]

Combining the above results for $n$ even and odd, we see that $|\ga_{ev}^{n,\nu}(z)|>0$ for all $z\in\T$ (In particlular, $\ga_{ev}^{2k,\nu}(z)>0$ for all $z\in\T$).  Hence by Wiener's lemma, its inverse $\gga = (\ga_{ev}^{n,\nu})^{-1}$ exists and  $\gga(1) = \frac{1}{\ga_{ev}^{n,\nu}(1)} = 1$. Moreover, by \eqref{eq:2-norm}, we have
\[
\|A_{\ga_{ev}^{n,\nu}}\|_2 = \max_{z\in\T}|\ga_{ev}^{n,\nu}(z)|=1
\] and
\[
\|A_\gga\|_2 =\max_{z\in\T}|\gga(z)|= \frac{1}{\min_{z\in\T}|\ga_{ev}^{n,\nu}(z)|}=\frac{2^{\lfloor \frac{n-1}{2}\rfloor+\nu}}{\sum_{j=0}^\nu{n/2+\nu\choose j}}.
\]
The exponential decay of the elements of $\gga$ in the case that $n=2k$, follows directly from Corollary~\ref{cor:1-norm-by-2-norm} since $\ga_{ev}^{2k,\nu}(z)>0$ for all $z\in\T$. In case $n=2k+1$, the exponential decay of the elements of $\gga$ follows from the weighted version of Wiener's lemma \cite{Greochenig2009}.
%
\end{proof}

\section{Proof of Theorem~\ref{thm:1-norm}}

\begin{proof}[Proof of Theorem~\ref{thm:1-norm}]
By \eqref{eq:2k} and  \eqref{eq:primalPositivity}, we have $\ga_{ev}^{2k,{\nu}}(z)>0$ for all $z\in\T$. Applying Corollary~\ref{cor:1-norm-by-2-norm} and Theorem~\ref{thm:pseudo-spline:inverse}, we have
\begin{equation}
\label{eq:gammaL}
|\gga_n|\le K \lambda^{|n|},\quad n\in\Z,
\end{equation}
where $K=\kappa\cdot\max\left\{1,\frac{(1+\sqrt{\kappa})^2}{2\kappa}\right\}$ with
\[
\kappa =\frac{\sup_{z\in\T}|\ga_{ev}^{2k,\nu}(z)|}{\inf_{z\in\T}|\ga_{ev}^{2k,\nu}(z)|}=\|A_\gga\|_2 = \frac{2^{k+{\nu}-1}}{\sum_{j=0}^{\nu} {k+{\nu} \choose j}} ,
\]
and $\lambda = q^{1/s}$ with
 \[
 q=(\sqrt{\kappa}-1)/(\sqrt{\kappa}+1), \quad s= \lfloor (k+{\nu})/2\rfloor.
 \]
 Thus by \eqref{eq:gammaL} we have
\[
\begin{aligned}
\|\gga\|_1
&= \sum_{n\in\Z} |\gga_n|
= |\gga_0|+2\sum_{n=1}^\infty |\gga_n|
\le K\Big(1+2\sum_{n=1}^\infty \lambda^n\Big)
\\&=K\Big(-1+2\sum_{n=0}^\infty \lambda^n\Big)
=K\frac{1+\lambda}{1-\lambda} = C(k,\nu).
\end{aligned}
\]

\end{proof}

\section*{Acknowledgments}
The authors thank Felipe Cucker for his support which initiated this paper.



\end{document}